\documentclass[reqno,12pt]{amsart}
\usepackage{amsthm}
\usepackage{amssymb,amsmath}
\usepackage{xypic}
\usepackage{hyperref}
\usepackage{bbm}
\usepackage{pbox}
\usepackage{cleveref}

\title[Characteristic classes of Higgs bundles]{Characteristic classes of Higgs bundles and Reznikov's theorem}
\author{Eric O. Korman}
\address{Department of Mathematics \\
The University of Texas at Austin \\
2515 Speedway, RLM 8.100 \\
Austin, TX 78712}
\email{ekorman@math.utexas.edu}
\urladdr{http://www.ma.utexas.edu/users/ekorman/}
\date{}

\renewcommand{\O}{\mathcal O}

\newcommand{\C}{\mathbb{C}}

\newcommand{\R}{\mathbb{R}}

\newcommand{\Z}{\mathbb{Z}}
\newcommand{\F}{\mathbb{F}}

\newcommand{\End}{\operatorname{End}}

\newcommand{\tr}{\operatorname{tr}}

\newcommand{\w}{\wedge}
\newcommand{\om}{\omega}
\newcommand{\Om}{\Omega}

\newcommand{\A}{\mathcal A}
\newcommand{\rk}{\operatorname{rk}}

\newcommand{\Hom}{\text{Hom}}
\newcommand{\ch}{\operatorname{ch}}

\renewcommand{\F}{\mathcal F}

\newtheorem{thm}{Theorem}

\newtheorem{defin}{Definition}
\newtheorem{prop}{Proposition}

\newcommand{\At}{\operatorname{At}}
\newcommand{\at}{\operatorname{at}}
\newcommand{\ah}{\operatorname{ah}}

\begin{document}
\maketitle

\begin{abstract}
We introduce Dolbeault cohomology valued characteristic classes of Higgs bundles over complex manifolds.  Flat vector bundles have characteristic classes lying in odd degree de Rham cohomology and a theorem of Reznikov says that these must vanish in degrees three and higher over compact K\"ahler manifolds.  We provide a simple and independent proof of Reznikov's result and show that our characteristic classes of Higgs bundles and the characteristic classes of flat vector bundles are compatible via the nonabelian Hodge theorem.
\end{abstract}

\section{Introduction}
Flat complex vector bundles over a manifold $X$ have Cheeger-Simons \cite{Cheeger} classes $CS_{2j+1} \in H^{2j+1}(X;\C/\Z)$, which, for algebraic manifolds, correspond to the Deligne Beilinson Chern classes \cite{Dupont,Rez, Soule}.  The imaginary part, $v_{2j+1}$, of $CS_{2j+1}$ lies in $H^{2j+1}(X;\R)$ and has a nice geometric description in terms of adjoint connections that fits into the theory of characteristic classes of foliated bundles \cite{BL, KT}. 

For $X$ a compact K\"ahler manifold, the nonabelian Hodge theorem of Simpson \cite{Simpson} gives a correspondence between semisimple flat vector bundles over $X$ and polystable Higgs bundles over $X$ with vanishing (rational) Chern classes.  A natural question is if there is a secondary characteristic class theory for Higgs bundles that, for polystable Higgs bundles with vanishing Chern classes, corresponds to the classes $v_{2j+1}$ via the nonabelian Hodge theorem and Hodge decomposition.  It turns out that the classes $v_{2j+1}$ necessarily vanish in degrees three and higher over compact K\"ahler manifolds.  This was a conjecture of Bloch \cite{Bloch} and was proved by Reznikov \cite{Rez}.

In this paper we first give a simple and independent proof of Reznikov's theorem using the nonabelian Hodge theorem of Simpson \cite{Simpson} and a Chern-Weil description of the flat characteristic classes $v_{2j+1}$, which appears in the work of Bismut and Lott \cite{BL}.  Differential form representatives of the characteristic classes are given in terms of an arbitrary hermitian metric on the vector bundle.  Because our proof shows the vanishing of the differential forms associated to the harmonic metric, we are able to define primitive forms for the representatives associated to an arbitrary hermitian metric.

We then define secondary characteristic classes for Higgs bundles over an arbitrary complex manifold $X$, which take values in the Dolbeault cohomology groups $H^{k+1,k}(X)$, $k \ge 0$.  In the case of a polystable Higgs bundle with zero Chern classes over a compact K\"ahler manifold, these classes vanish for $k \ge 1$ and for $k=0$ have real part equal to $-\frac{1}{2}v_1$ under the Hodge decomposition.  The classes are defined simply as
\[
\frac{1}{k!} \left(\frac{i}{2\pi}\right)^k \tr(\At(E)^k \theta) \in H^{k+1,k}(X;\C),
\]
where 
\[
\At E \in H^{1,1}(X;\End E) \simeq H^1(\Om_X^1\otimes \mathcal O(\End E))
\] 
is the Atiyah class of $E$ and 
\[
\theta \in H^{1,0}(X;\End E)\simeq H^0(X;\Om_X^1\otimes \mathcal O(\End E))
\]
is the Higgs field.  We will see that the classes can also be defined as transgression forms for scalar Atiyah classes.

We remark that characteristic classes for families of polystable flat Higgs bundles appear in \cite{Bis}.

\section*{Acknowledgments}
The author thanks Jonathan Block and Tony Pantev for helpful discussions regarding this work.

\section{Reznikov's Theorem}
In this section we briefly review a Chern-Weil construction of the characteristic classes of flat vector bundles in odd degree de Rham cohomology and the nonabelian Hodge theorem.  We then show how these two combine to give a quick proof of Reznikov's result \cite{Rez}. 

For a smooth manifold $X$ and vector bundle $E\to X$, $\A^\bullet(X; E)$ denotes the space of $E$-valued differential forms.  If $X$ is a complex manifold then $\A^{p,q}(X;E)$ denotes the $E$-valued forms in bidegree $(p,q)$.

\subsection{Characteristic classes of flat vector bundles}
Let $X$ be an arbitrary smooth manifold and $(E,\nabla) \to X$ a flat complex vector bundle.  In \cite{BL}, the classes $v_{2j+1}$ are defined as follows.  First choose an arbitrary hermitian metric $h$ on $E$ and let $\nabla^*$ denote the adjoint connection, i.e. $\nabla^*$ is the connection on $E$ defined by
\[
dh(\psi_1,\psi_2) = h(\nabla \psi_1, \psi_2) + h(\psi_1,\nabla^* \psi_2), ~~ \psi_1,\psi_2 \in \Gamma(X;E).
\]
Note that $\nabla^* = \nabla$ for some $h$ if and only if the holonomy representation of $\pi_1(X)$ defining the flat bundle is unitary.  Now define
\[
\om(E,\nabla,h) = \frac{1}{2}(\nabla^* - \nabla) \in \A^1(X;\End E)
\]
and, for $j \ge 0$,
\begin{equation}
v_{2j+1}(E,\nabla,h) = (2\pi i)^{-j} \tr \om(E,\nabla,h)^{2j+1} \in \A^{2j+1}(X;\R). \label{cDEF}
\end{equation}
These forms are closed and their cohomology classes are independent of the choice of $h$ \cite{BL}.  Indeed, if $h_t, t \in [0,1]$ is a smooth family of metrics then an explicit $2j$-form $Tv_{2j}(E,\nabla,h_0,h_1)$ satisfying
\[
d Tv_{2j}(E,\nabla,h_0,h_1) = v_{2j+1}(E,\nabla,h_1) - v_{2j+1}(E,\nabla,h_0)
\]
is given by \cite{BL}
\begin{equation}
Tv_{2j}(E,\nabla,h_t) = \frac{1}{2}(2j+1)(2i\pi)^{-j} \tr \int_0^1 h_t^{-1} \frac{\partial h_t}{\partial t} \om(E,\nabla,h_t)^{2j}  dt  \label{transgression},
\end{equation}
where above we are viewing $h_t$ as a map $E\to \overline E^*$.

\begin{defin}
Let $v_{2j+1}(E,\nabla) \in H^{2j+1}(X;\R)$ be the cohomology class of $v_{2j+1}(E,\nabla,h)$, where $h$ is an arbitrary hermitian metric on $E$.
\end{defin}
We will often write $v_{2j+1}(E)$ for $v_{2j+1}(E,\nabla)$.  These characteristic classes have the following properties, whose proofs can be found in \cite{BL}.  
\begin{prop}
If $E_1, E_2 \to X$ are both flat then
\begin{enumerate}
\item $v_k(E_1\oplus E_2) = v_k(E_1) + v_k(E_2)$.
\item $v_k(E_1\otimes E_2) = (\rk E_1)v_k(E_2) + (\rk E_2) v_k(E_1)$.
\end{enumerate}
\end{prop}

The classes have another geometric description, used by Reznikov \cite{Rez} (who calls these volume regulators), and which fits into the framework of Kamber and Tondeur's characteristic classes for foliated bundles \cite{KT}.  Let $\tilde X$ denote the universal cover of $X$ and let $\rho: \pi_1(X) \to GL(n,\C)$ denote the holonomy representation of the flat bundle $E$.  Consider the fiber bundle
\[
\F = \tilde X \times_\rho GL(n,\C)/U(n),
\]
whose smooth sections correspond to the space of hermitian metrics on $E$.  For each $k$, we define
\begin{gather*}
\psi_k \in \Lambda^k \mathfrak u(n)^\perp \simeq \Lambda^k (i\mathfrak u(n))^*, \\
\psi_k(A_1,\ldots,A_k) = \operatorname{Alt}\tr(A_1\cdots A_k).
\end{gather*}
This determines a closed $GL(n,\C)$ invariant $k$-form on $GL(n,\C)/U(n)$, which further gives a closed $k$-form $\tilde \psi_k$ on $\F$.  Then given any metric $h\in \Gamma(X;\F)$, $h^*\tilde \psi_k$ is a closed $k$-form on $X$ whose cohomology class is independent of $h$ since $GL(n,\C)/U(n)$ is contractible.  Straightforward computations show that, up to constant factors, this is the class $v_k(E)$ \cite{BL}.

\subsection{The nonabelian Hodge theorem}
A Higgs bundle $(E,\theta)$ over a complex manifold $X$ consists of a holomorphic vector bundle $(E, \bar\partial_E)\to X$ together with $\theta \in \A^{1,0}(X;\End E)$ such that
\[
[\bar\partial_E, \theta] = 0, ~~ \theta \w \theta = 0,
\]
where $[\cdot,\cdot]$ will always denote the supercommutator.  For $X$ a compact K\"ahler manifold, the nonabelian Hodge theorem \cite{Simpson} says that there is a 1-to-1 correspondence
\[
\left(\text{\parbox{4cm}{\centering polystable Higgs bundles on $X$ with vanishing rational Chern classes}}\right)\longleftrightarrow \left(\text{\parbox{4cm}{\centering semi-simple flat vector bundles on $X$}}\right).
\]
To go from a Higgs bundle to a flat bundle requires the existence of a Hermitian-Yang-Mills metric $h$.  The flat connection is then given by
\[
\nabla = \nabla_h + \theta + \theta^*,
\]
where $\nabla_h$ denotes the Chern connection (i.e. the unique connection preserving $h$ and having $(0,1)$ part equal to $\bar\partial_E$) and $\theta^*$ is defined by
\[
h(\theta \psi_1, \psi_2) = h(\psi_1,\theta^* \psi_2), ~~ \psi_1,\psi_2\in \Gamma(X;E).
\]

\subsection{Proof of Reznikov's theorem}
We can now give a quick proof of Reznikov's theorem \cite{Rez}:
\begin{thm}
For $X$ a smooth complex projective variety, the Chern classes in Deligne cohomology are torsion in degrees 2 and higher.
\end{thm}
Using properties of Borel regulators, Reznikov shows that this theorem is implied by the following proposition, of which we give an independent proof.
\begin{prop}\label{propCVanish}
If $(E,\nabla) \to X$ is a semisimple flat vector bundle over a compact K\"ahler manifold, then $v_{2j+1}(E) = 0$ for $j \ge 1$.
\end{prop}
\begin{proof}
By the nonabelian Hodge theorem, we have
\begin{equation*}
\nabla = \nabla_h + \theta + \theta^*,
\end{equation*}
where $h$ is the harmonic metric.  Now, $\nabla_h$ preserves $h$ and $\theta + \theta^*$ is hermitian (whereas to preserve the metric it would have to be skew-hermitian) so that the adjoint of $\nabla$ with respect to $h$ is
\begin{equation*}
\nabla^* =\nabla_h - \theta - \theta^*.
\end{equation*}
Thus
\[
\om(E,h) = -(\theta + \theta^*).
\]
But since $\theta^2 = 0 = (\theta^*)^2$ we have, for $j > 1$,
\begin{align*}
\tr(\om(E,h)^{2j+1}) &= -\tr\left((\theta + \theta^*)^{2j+1}\right) \\
&= -\tr((\theta\theta^* +\theta^*\theta)^j (\theta + \theta^*)) \\
&= -\tr((\theta\theta^* \cdots \theta\theta^*) \theta) - \tr((\theta^*\theta \cdots \theta^*\theta) \theta^*) \\
&= -\tr(\theta(\theta\theta^* \cdots \theta\theta^*)) - \tr(\theta^*(\theta^*\theta \cdots \theta^*\theta)) \\
&= 0,
\end{align*}
where in the second to last line we used the fact that the trace vanishes on commutators.  Thus $v_{2j+1}(E) = 0$ by \cref{cDEF}.
\end{proof}

Reznikov gives two proofs of the above proposition.  One uses a result of Sampson on harmonic maps of K\"ahler manifolds into symmetric spaces \cite{Sampson} and the other uses the deformation theory of Higgs bundles \cite{Simpson}.

\subsection{A natural primitive form}
For $(E,\nabla)$ a semi-simple flat vector bundle over a compact manifold, our proof of \cref{propCVanish} shows that not only the cohomology class $v_{2j+1}(E) \in H^{j+1,j}(X)$ vanishes, but that its representative $v_{2j+1}(E,h) \in \A^{j+1,j}(X)$ vanishes when $h$ is the harmonic metric.  Therefore if $k$ is any hermitian metric we can use the transgression \cref{transgression} to get a primitive for $v_{2j+1}(E,k)$:
\[
d Tv_{2j}(E,\nabla, h + t(k-h)) = v_{2j+1}(E,k).
\]

\section{Characteristic classes of Higgs bundles in Dolbeault cohomology}
In this section we define secondary characteristic classes for Higgs bundles.  We first recall the definitions and properties of Atiyah classes.
\subsection{Atiyah classes}
Let $(E,\bar\partial_E)$ be a holomorphic vector bundle over a complex manifold $X$ (for now $X$ need not be K\"ahler).  Associated to $E$ is its Atiyah class \cite{Atiyah}
\[
\At E \in H^1(\Om_X^1\otimes \mathcal O(\End E)).
\]    
A Dolbeault representative of the Atiyah class in $H^{1,1}(X;\End E)$, which we also denote by $\At E$, is defined as follows.  Let $\nabla$ be a connection on $E$ whose (0,1) part is $\bar\partial_E$ (for example, $\nabla$ can be taken to be the Chern connection associated to some hermitian metric).  Then $\At E$ is represented by the bidegree (1,1) part of the curvature of $\nabla$ \cite{CSX, ABST}.

The Atiyah-Chern character is defined by
\[
\ch E  = \tr \exp\left(\frac{i}{2\pi} \At E\right) \in \bigoplus_k H^{k,k}(X)
\]
and the $k$th scalar Atiyah class, denoted $\at_k (E)$, is defined to be the bidegree $(k,k)$ part of $\ch E$, i.e.
\[
\at_k(E) = \frac{1}{k!}\left(\frac{i}{2\pi}\right)^k \tr((\At E)^k) \in H^{k,k}(X).
\]
In the case that $X$ is compact and K\"ahler, the Atiyah-Chern character corresponds, via the Hodge decomposition, to the usual Chern character in de Rham cohomology.  We remark that even if the scalar Atiyah classes all vanish, the full Atiyah class, $\At E$, may still be nonzero.

\begin{prop}\label{Atprops}
The Atiyah classes satisfy the following properties.
\begin{enumerate}

\item If $f: X \to Y$ is a holomorphic map and $E\to Y$ a holomorphic vector bundle, then
\[
\At(f^* E) = f^*\At E \in H^{1,1}(X;f^* \End E).
\]

\item If
\[
0 \to E' \to E \to E'' \to 0
\]
is a short exact sequence of holomorphic vector bundles, then
\[
\ch(E) = \ch(E') + \ch(E'').
\]

\item If $E'$ and $E''$ are two holomorphic vector bundles, then
\begin{gather*}
\At(E'\otimes E'') = \At E' \otimes\mathbbm 1_{E''} + \mathbbm 1_{E'}\otimes \At(E'') \\
\ch(E'\otimes E'') = \ch E' \ch E''.
\end{gather*}
\end{enumerate}
\end{prop}
\begin{proof}
The first statement follows from the fact that if $\nabla$ is a connection on $E$ compatible with $\bar\partial_E$ then $f^*\nabla$ is compatible with $\bar\partial_{f^*\nabla}$ and the curvatures are related by $(f^*\nabla)^2 = f^* \nabla^2$.  

To prove (2), let $s: E'' \to E$ be a $C^\infty$ splitting of the sequence and define
\[
\bar\partial s := \bar\partial_E \circ s - s\circ \bar\partial_{E''} \in \A^{0,1}(X; \Hom(E'', E')).
\]
Choose connections $\nabla'$ and $\nabla''$ on $E'$ and $E''$ that are compatible with their respective holomorphic structures.  Then, under the decomposition $E \simeq E' \oplus E''$ provided by $s$, the connection
\[
\nabla = \left(\begin{array}{cc}
\nabla' & \bar\partial s \\
0 & \nabla''
\end{array}\right)
\]
is easily checked to be compatible with the holomorphic structure on $E$ and
\[
\tr \exp\nabla^2 = \tr \exp (\nabla')^2 + \tr \exp(\nabla'')^2,
\]
from which the result follows.

Finally, for (3) observe that if $\nabla'$ and $\nabla''$ are compatible connections on $E'$ and $E''$, respectively, then $\nabla' \otimes \mathbbm 1_{E''} + \mathbbm 1_{E'} \otimes \nabla''$ is a compatible connection on $E' \otimes E''$ with curvature $(\nabla')^2 \otimes \mathbbm 1_{E''} + \mathbbm 1_{E'} \otimes (\nabla'')^2$.

\end{proof}

\subsubsection{Transgression forms}\label{transforms}
For a connection $\nabla$ compatible with $\bar\partial_E$, let $\at_k(E,\nabla) = \frac{1}{k!} \left(\frac{i}{2\pi}\right)^k\tr((\nabla^2)_{(1,1)})^k\in \A^{k,k}(X)$ denote the specific representative of $\at_k(E)$ determined by $\nabla$.  Given two compatible connections $\nabla_0, \nabla_1$, we can explicitly define a transgression form $\at_k(E,\nabla_0, \nabla_1) \in \A^{k,k-1}(X)$ satisfying
\[
\bar\partial \at_k(E,\nabla_0,\nabla_1) = \at_k(E,\nabla_1) - \at_k(E,\nabla_0).
\]
These forms are defined by
\begin{align*}
\at_k(E,\nabla_0,\nabla_1) &= \\
&\hspace{-45pt}\left(\frac{1}{k!}\left(\frac{i}{2\pi}\right)^k  \int_0^1 \tr\left(\left(\nabla_0 + t(\nabla_1 - \nabla_0) + dt\otimes\frac{\partial}{\partial t}\right)^2\right)^k\right)_{(k,k-1)},
\end{align*}
where the subscript $(k,k-1)$ means to take the part in bidegree $(k,k-1)$.

\subsection{Characteristic classes of Higgs bundles}
Now let $(E,\theta)$ be a Higgs bundle over the complex manifold $X$.  The Higgs field $\theta$ defines a class in $H^{1,0}(X;\End E)$ and we make the following definition:
\begin{defin}
For $k \ge 0$, let
\[
\operatorname{ah}_k(E,\theta) = \frac{1}{k!} \left(\frac{i}{2\pi}\right)^k \tr \left((\At E)^k \theta\right) \in H^{k+1,k}(X;\C).
\]
\end{defin}

\begin{prop}
The characteristic classes of Higgs bundles satisfy the following properties.
\begin{enumerate}
\item If $f: X \to Y$ is a holomorphic map and $(E,\theta) \to Y$ a Higgs bundle, then 
\[
\ah_k(f^*E,f^*\theta) = f^* \ah_k(E,\theta),
\]
for all $k \ge 0$.

\item If
\[
0 \to (E',\theta') \to (E,\theta) \to (E'',\theta'')\to 0
\]
is a short exact sequence of Higgs bundles, then
\[
\ah_k(E,\theta) = \ah_k(E',\theta') + \ah_k(E'',\theta''),
\]
for all $k \ge 0$.

\item If $(E',\theta')$ and $(E'',\theta'')$ are two Higgs bundles, then
\begin{align*}
\ah_k(E'\otimes E'', &\theta'\otimes \mathbbm 1_{E''} + \mathbbm 1_{E'} \otimes \theta'') \\ 
&= \sum_{j=0}^k \left(\ah_j(E',\theta') \at_{k-j}(E'') + \at_j(E')\ah_{k-j}(E'',\theta'') \right),
\end{align*}
for all $k\ge 0$.

\end{enumerate}
\end{prop}
\begin{proof}
The first statement follows from \cref{Atprops}(1).  Part (2) follows from the proof of \cref{Atprops}(2) and the fact that, under the splitting $E = E' \oplus E''$, the Higgs field on $E$ takes the form
\[
\theta = \left(\begin{array}{cc}
\theta' & * \\
0 & \theta''
\end{array}\right).
\]

The last statement is a straightforward computation using \cref{Atprops}(3).
\end{proof}

We will now see that, up to a constant factor, these classes are certain transgression forms from \cref{transforms}.  If $\nabla$ is any connection compatible with $\bar\partial_E$ then so is $\nabla + \theta$.  Further, since $[\bar\partial_E, \theta] = 0$, the (1,1) parts of the curvatures of $\nabla$ and $\nabla+\theta$ are the same.  Thus
\[
\bar\partial \at_k(E,\nabla,\nabla + \theta) = \at_k(E,\nabla) - \at_k(E,\nabla + \theta) = 0.
\]

\begin{prop}
For all $k \ge 0$,
\[
\ah_k(E,\theta) = 2\pi i \at_{k+1}(E,\nabla,\nabla+\theta).
\]
\end{prop}
\begin{proof}
Since $\theta\w\theta = 0$, we have
\[
\left(\nabla + t\theta + dt \otimes \frac{\partial}{\partial t}\right)^2 = \nabla^2 + t[\nabla,\theta] + dt\w\theta.
\]
Since $\theta$ is holomorphic and $\nabla$ is a compatible connection, $[\nabla,\theta]$ has bidegree $(2,0)$. Thus the only pieces that contribute when computing $\at_{k+1}(E,\nabla,\nabla+\theta)$ are $(\nabla^2)_{(1,1)}$ and $dt \w \theta$, giving us
\begin{align*}
2\pi i\at_{k+1}(E,\nabla,\nabla + \theta) &= \frac{-1}{(k+1)!} \left(\frac{i}{2\pi}\right)^k \int_0^1 \tr\left((\nabla^2)_{(1,1)} + dt \w \theta\right)^{k+1} \\
&= \frac{-1}{(k+1)!} \left(\frac{i}{2\pi}\right)^k \hspace{-5pt} \int_0^1 \hspace{-5pt} \tr((k \hspace{-2pt} +\hspace{-2pt}1)(\nabla^2)_{(1,1)}^k \w dt \w \theta) \\
&= \frac{1}{k!}\left(\frac{i}{2\pi}\right)^k \tr\left((\nabla^2)_{(1,1)}^k \theta\right) \\
&= \ah_k(E,\theta).
\end{align*}
\end{proof}

\subsection{The case of polystable Higgs bundles with vanishing Chern classes}
We now show that these classes correspond to the classes $v_{2j+1}$ when the Higgs bundle corresponds to a flat vector bundle under the nonabelian Hodge theorem.  In particular, in this case the classes $\ah_k$ vanish for $k\ge 1$. 

\begin{prop}
Suppose $(E,\theta) \to X$ is a polystable Higgs bundle with vanishing Chern classes over a compact K\"ahler manifold $X$ and let $\nabla$ denote the corresponding flat connection.  Then
\begin{enumerate}
\item  Using the Hodge decomposition to view $\ah_0(E,\theta)$ as a class in $H^1(X;\C)$, we have
\[
\operatorname{Re} \ah_0(E,\theta) = \frac{1}{2} (\ah_0(E,\theta) + \overline{\ah_0(E,\theta)})  = -\frac{1}{2} v_1(E,\nabla).
\]

\item $\ah_k(E,\theta) = 0$, for $k \ge 1$.
\end{enumerate}
\end{prop}
\begin{proof}
As in the proof of \cref{propCVanish}, if $h$ is the Hermitian-Yang-Mills metric then
\begin{gather}
\nabla = \nabla_h + \theta + \theta^* \label{flatconn}\\
\nabla^* = \nabla_h - \theta - \theta^*\nonumber
\end{gather}
so that
\begin{gather*}
v_1(E,\nabla,h) = \frac{1}{2}\tr(\nabla^* - \nabla) = -\tr(\theta + \theta^*) \\
= -\tr \theta - \overline{\tr\theta} = -\ah_0(E,\theta,h) - \overline{\ah_0(E,\theta,h)}.
\end{gather*}
Since $\ah_0(E,\theta,h)$ is a $\bar\partial$-closed $(1,0)$ form it is actually harmonic (since we trivially have $\bar\partial^* \ah_0(E,\theta,h)$).  Thus taking cohomology and using the Hodge decomposition gives (1).

The proof of  (2) is similar to that of \cref{propCVanish}.  From \cref{flatconn}, we have
\[
0 = (\nabla)^2 = F_h + [\nabla_h, \theta + \theta^*] + [\theta, \theta^*].
\]
Since $[\bar\partial_E,\theta] = 0$, the $(1,1)$ part of the above equation gives us
\[
F_h = -[\theta,\theta^*].
\]
Thus $\At(E)$ is represented by the differential form $-[\theta,\theta^*]$.
Then for $k \ge 1$,
\[
\ah_k(E,\theta,h) = (-1)^k\tr([\theta,\theta^*]^k \theta) = (-1)^k \tr((\theta\theta^* \cdots \theta \theta^*) \theta) = 0,
\]
since the trace vanishes on commutators and $\theta\w\theta = 0$.
\end{proof}
 If we drop either of the conditions (vanishing Chern classes or polystability), then the classes $\ah_k$ do not vanish in general.  For example, if $E\to X$ is a non-trivial line bundle and $\theta$ represents a non-zero class in $H^{1,0}(X)$, the pair $(E,\theta)$ is a stable Higgs bundle with $\ah_1(E,\theta) = \at_1(E) \theta$ non-zero in general.  For an example of a non-polystable Higgs bundle with vanishing Chern classes and non-vanishing $\ah_1$, we take \footnote{this example is drawn in part from \url{http://mathoverflow.net/q/161894/4622}}
 \[
 E = p^* \O(-1) \oplus p^* \O(1) \to \C P^1 \times \C/\Z^2,
 \]
 where $p: \C P^1 \times \C/\Z^2 \to\C P^1$ is the projection, and
 \[
 \theta = \left(\begin{array}{cc}
 \mathbbm 1_{p^*\O(-1)} & 0 \\
 0 & - \mathbbm 1_{p^*\O(1)}
 \end{array}\right) dz,
 \]
 where $dz$ represents a generator of $H^{1,0}(\C/\Z^2)$.  The Chern class of $\O(\pm 1)$ is represented by $\pm \om$, where $\om$ denotes the Fubini-Study form on $\C P^1$.  Then $\ch E = \rk E = 2$ (indeed, $E$ is actually topologically trivial) and the Atiyah class of $E$ is represented by
 \[
 \At(E) = \left(\begin{array}{cc}
 \om & 0 \\
 0 & -\om
 \end{array}\right)
 \]
 so that $\ah_1(E,\theta) = 2\om \w dz \ne 0$.  Of course $(E,\theta)$ is not polystable (it is the direct sum of stable Higgs bundles of different slopes).

\bibliographystyle{alpha}
\bibliography{biblio}

\end{document}